\documentclass{amsart}
\numberwithin{equation}{section}
\usepackage{amssymb}
\usepackage{amsmath}
\usepackage{amsthm}
\usepackage{amscd}
\usepackage{amsfonts}
\usepackage{ascmac}
\usepackage[usenames]{color}
\usepackage{graphics}

\theoremstyle{plain}
\newtheorem{theorem}{Theorem}[section]
\newtheorem{proposition}[theorem]{Proposition}

\newtheorem{lemma}[theorem]{Lemma}

\newtheorem{definition}[theorem]{Definition}

\newtheorem{remark}[theorem]{Remark}

\newcommand{\bfx}{{\mathbf x}}

\newcommand{\bfo}{{\mathbf o}}

\newcommand{\bfR}{{\mathbf R}}
\newcommand{\bfZ}{{\mathbf Z}}

\newcommand{\barj}{{\overline j}}

\newcommand{\barpartial}{{\overline \partial}}

\newcommand{\mapright}[1]{\smash{\mathop{   \hbox to 0.7cm{\rightarrowfill}}
  \limits^{#1}}}

\newcommand{\Ric}{\operatorname{Ric}}

\newcommand{\Fut}{\mathrm{Fut}}
\newcommand{\Vol}{\mathrm{Vol}}

\title{Moment polytopes on Sasaki manifolds and volume minimization}
\author{Akito Futaki}
\address{Yau Mathematical Sciences Center, Tsinghua University, Haidian district, Beijing 100084, China}
\email{futaki@tsinghua.edu.cn}

\date{January 26, 2022}

\begin{document}

\begin{abstract}
We show that transverse coupled K\"ahler-Einstein
metrics on toric Sasaki manifolds arise as a critical point of a volume functional. As a preparation for the proof, 
we re-visit the transverse moment polytopes and contact moment polytopes 
under the change of Reeb vector fields. Then we apply it to a coupled version of the volume minimization by Martelli-Sparks-Yau. This is done assuming the Calabi-Yau condition of the K\"ahler cone, and the non-coupled case leads to 
a known existence result of a transverse K\"ahler-Einstein metric and a Sasaki-Einstein metric, but the coupled case requires an assumption related to Minkowski sum to obtain transverse coupled K\"ahler-Einstein metrics.
\end{abstract}

\maketitle

\section{Introduction}
Sasaki-Einstein metrics drew much attention from theoretical physics and mathematics during last two decades.
The first breakthrough was an irregular example found in physics literature 
by Gauntlett, Martelli, Sparks and Waldrum
\cite{GMSW04}. Then in the toric case the existence was proven in our paper 
\cite{FOW} using volume minimization
of Martelli, Sparks and Yau
\cite{MSY1}, \cite{MSY2}. More recently it has been shown that the existence is equivalent
to a notion called $K$-stability by Collins and Sz\'ekelyhidi
\cite{CollinsSzeke18JDG}, \cite{CollinsSzeke19GT}. Sasaki
manifolds are characterized by two K\"ahler structures, one on the Riemannian cone and 
the other 
on the local orbit spaces of the one parameter group of transformations, which we call the {\it Reeb flow}, 
generated by the Reeb vector field.
In fact, the existence of a Sasaki-Einstein metric is equivalent to the existence of a Ricci-flat K\"ahler metric on
the K\"ahler cone, and also equivalent to the existence of a transverse K\"ahler-Einstein metric
of positive scalar curvature on local orbit spaces of the Reeb flow. Therefore there are two possible 
extensions of these studies, one on the K\"ahler cone and the other on the K\"ahler local orbit
spaces of the Reeb flow. In \cite{deBorbonLeg20}, de Borbon and Legendre used
the volume minimization argument to prove the existence on toric K\"ahler cone manifolds of Ricci-flat K\"ahler cone metrics with
cone angle along the boundary invariant divisors without assuming the Calabi-Yau condition
of the K\"ahler cone. This Calabi-Yau condition will be explained in the paragraph 
after Proposition \ref{thm2} below. 
The purpose of this paper is to study the possibility to prove the existence on toric Sasaki manifolds of
transverse coupled K\"ahler-Einstein metrics in the sense of Hultgren and Witt Nystrom
\cite{HultgrenWittNystrom18} assuming the Calabi-Yau condition of the K\"ahler cone
by using the volume minimization argument of Martelli-Sparks-Yau.
Our study shows that 
the non-coupled transverse K\"ahler-Einstein metric recovers the toric Sasaki-Einstein metrics as in \cite{FOW}
but the coupled case requires an additional Minkowski sum assumption to obtain transverse coupled K\"ahler-Einstein metrics.

The transverse coupled K\"ahler-Einstein metrics are defined as follows. 
A Sasaki manifold $S$ is determined by contact distribution $D$, pseudo-convex $CR$-structure $J$ on $D$
and Reeb vector field $\xi$. The pseudo-convex $CR$-structure determines K\"ahler structures on
the local orbit spaces of the Reeb flow. Differential forms on $S$ 
obtained by pulling back from those local orbit spaces are called {\it basic forms}. Naturally
$\partial$ and $\barpartial$ operators can be considered to operate on basic forms, which we denote by
$\partial_B$ and $\barpartial_B$, and we obtain Dolbeault theory, Hodge theory and Chern-Weil theory for basic
forms.
Suppose that the basic first Chern class $c_1^B(S)$ is positive, i.e. represented by a real closed positive
$(1,1)$-basic form, and that we are given
a decomposition
\begin{equation}\label{decomp}
 2\pi c_1^B(S) = \gamma_1 + \cdots + \gamma_k
\end{equation}
of $2\pi c_1^B(S)$ into a sum of basic K\"ahler classes $\gamma_\alpha$. 
Basic K\"ahler metrics $\omega_\alpha \in \gamma_\alpha$ are called
{\it transverse coupled K\"ahler-Einstein metrics} if 
\begin{equation}\label{coupled}
\rho^T(\omega_1) = \cdots = \rho^T(\omega_k) = \sum_{\beta=1}^k \omega_\beta
\end{equation}
where
$$ \rho^T(\omega_\alpha) = - i \partial_B \barpartial_B \log \omega_\alpha^m$$
is the transverse Ricci form of $\omega_\alpha$. Naturally, by the Chern-Weil theory,
$$2\pi c_1^B(S) = [\rho^T(\omega_\alpha)]_B$$ 
where $[\cdot]_B$ denotes a basic cohomology class.

As a preparation, 
we study the relation of transverse moment map image and the contact moment map image, 
and how 
the decompositions of the basic first Chern class induce the Minkowski sum decompositions of the image of the contact moment map. 
The linkage of transverse moment map and the contact moment map is played by the conditions
$c_1^B(S) > 0$ and $c_1(D) =0$ where $D$ is the contact distribution with complex structure $J$.
As will be shown in Lemma \ref{SE2} these two conditions imply that 
$$c_1^B(S) = \tau [d\eta_\xi]_B$$
for some positive constant $\tau$ where $\eta_\xi$ is the contact form with respect to the Reeb vector
field $\xi$. Since the transverse moment map is with respect to the basic K\"ahler class
$c_1^B(S)$ and the contact moment map is with respect to the contact form $\eta_\xi$ we can compare
the two moment maps.
The result we obtain about the comparison
is stated as follows.
\begin{proposition}\label{thm2}
Let $(S,D,J,\xi)$ be a Sasaki manifold such that $c_1^B(S) > 0$ and $c_1(D) = 0$. Suppose that a real compact torus $T$ acts effectively on $S$ preserving
$(D,J,\xi)$ and that the Lie algebra $\mathfrak t$ of $T$
contains $\xi$ (but we do not need to assume $S$ is toric in this proposition).
Suppose also that we are given a decomposition \eqref{decomp}.
\begin{enumerate}
\item[(1)]
There is a unique point $\bfo$, which we call {\it the origin}, in the image 
$\mathcal P_\xi \subset \{p \in \mathfrak t^\ast\,|\, \langle p, \xi \rangle = 1\}$
of the contact moment map and a Minkowski sum decomposition
\begin{equation}\label{Mink0}
 \mathcal P_\xi = \mathcal P_{\xi,1} + \cdots + \mathcal P_{\xi,k}
 \end{equation}
into the sum of convex polytopes $\mathcal P_{\xi,\alpha} \subset \mathcal P_\xi,$
where we regard the hyperplane $\{p \in \mathfrak t^\ast\,|\, \langle p, \xi \rangle = 1\}$
as a vector space by choosing the origin $\bfo$ to be zero,
such that if there are transverse coupled K\"ahler-Einstein metrics
then the sum of the barycenters of $\mathcal P_{\xi,\alpha}$ lies at the origin $\bfo$.
\item[(2)] The Minkowski sum decomposition in {\rm (1)} is unique up to translations
of $\mathcal P_{\xi,\alpha}$ to $\mathcal P_{\xi,\alpha} + c_\alpha$ with $c_\alpha \in \mathfrak t^\ast$
such that $\sum_{\alpha=1}^k c_\alpha= \bfo$.
\item[(3)]
The Minkoswski sum decomposition of $\mathcal P_\xi$ in {\rm (1)} determines a Minkowski sum decomposition of the contact moment cone $\mathcal C_\xi$ 
\begin{equation}\label{Mink1} \mathcal C_\xi = \mathcal C_{\xi,1} + \cdots + \mathcal C_{\xi,k} \end{equation}
into the sum of cones $\mathcal C_{\xi,\alpha} \subset \mathfrak t^\ast$ in such a way that the intersection of $\mathcal C_{\xi,\alpha}$ with $\mathcal P_\xi$ is $\mathcal P_{\xi,\alpha}$.
\end{enumerate}
\end{proposition}
\noindent
The origin $\bfo$ in fact corresponds to the zero of the transverse moment polytope as the proof shows.

A Sasaki manifold $S$ of dimension $2m+1$ 
is said to be {\it toric} if its K\"ahler cone $C(S)$ is toric. Thus 
a real compact torus $T$ of dimension $m+1$ acts
effectively on $S$ preserving 
the contact distribution $D$, the pseudo-convex $CR$-structure $J$ on $D$
and the Reeb vector field $\xi$, and the Lie algebra $\mathfrak t$ of $T$
contains $\xi$ where the elements of $\mathfrak t$ are identified
with vector fields on $S$. 
Let 
$$ \mathcal C = \{ p \in \mathfrak t^\ast\backslash \{o\}\,|\, \langle p, \ell_a\rangle \ge 0,\ a = 1, \cdots, d\}$$
be the moment cone of $C(S)$, which is a convex rational polyhedral cone, 
where $\ell_a \in \mathfrak t$ such that $2\pi \ell_1, \cdots, 2\pi\ell_d$ are primitive elements of 
the kernel $\Lambda$ of $\exp : \mathfrak t \to T$. 
For a compact toric Sasaki manifold $S$ we have the following equivalent conditions, 
c.f. \cite{CFO}, Theorem 1.2:
\begin{enumerate}
\item $c_1^B(S) > 0$ and $c_1(D) =0$.
\item There is a rational vector $\gamma \in \mathfrak t^\ast$ such that 
$$ \langle\gamma,\xi\rangle = -m-1 \quad \text{and}\quad \ell_a(\gamma) = -1\ \text{for}\ a=1,\cdots, d.$$
\item The power of the canonical line bundle $K^{\otimes\ell}_{C(S)}$ of the cone $C(S)$ is a trivial line bundle for some integer $\ell$. 
\end{enumerate}
Because of (c) we call these equivalent conditions {\it Calabi-Yau condition of the K\"ahler cone}. 
The condition (b) appeared in \cite{MSY1} as (2.57), (2.60).
The existence of $-\gamma$ is also known in algebraic geometry of toric varieties,
see \cite{CoxLittleSchenck}, Theorem 4.2.8. The paper \cite{deBorbonLeg20} also gives an account from
the broader view points of what they call {\it angle cone}.

\begin{theorem}\label{thm3}
Let $S$ be a toric Sasaki manifold satisfying Calabi-Yau condition of the K\"ahler cone. 
Then, in Proposition \ref{thm2}, we can take
$$\bfo =- \frac1{m+1}\gamma.$$ 
\end{theorem}

Using Proposition \ref{thm2} and Theorem \ref{thm3} we apply the volume minimization argument of
Martelli-Sparks-Yau in the following setting.
Let $S$ be a toric Sasaki manifold satisfying Calabi-Yau condition of the K\"ahler cone.
We regard
\begin{eqnarray}
 \Xi_\bfo 
&:=& \{ \xi' \in \mathcal C^\ast \subset \mathfrak t\ |\ \langle \xi',\bfo \rangle = 1 \} \\
&=& \{ \xi' \in \mathcal C^\ast \subset \mathfrak t\ |\ \langle \xi',\gamma \rangle = -m-1 \}\nonumber
\end{eqnarray}
as the space of Reeb vector fields satisfying the Calabi-Yau conditions of the K\"ahler cone. 


Let $\gamma_1,\ \cdots,\ \gamma_k$ be basic K\"ahler classes with respect to the Reeb vector field $\xi$. But we do not assume
$ c_1^B(S) = (\gamma_1 + \cdots + \gamma_k)/2\pi$
for the moment.
Let $\mathcal P_{\xi,1},\ \cdots,\ \mathcal P_{\xi,k}$ be compact convex polytopes 
corresponding to $\gamma_1,\ \cdots,\ \gamma_k$, which are assumed to be subsets 
in the contact moment polytope $P_{\xi}$ of $S$, 
and $\mathcal C_{\xi,1},\ \cdots,\ \mathcal C_{\xi,k}$ be convex polyhedral cones in the contact moment
convex cone $\mathcal C_\xi$ of the K\"ahler cone $C(S)$ of $S$
such that $\mathcal P_{\xi,\alpha} = \mathcal C_{\xi, \alpha} \cap P_{\xi}$. 

Choose $\xi' \in \Xi_\bfo$, and set for $\alpha = 1,\ \cdots,\ k$
$$ \mathcal P_{\xi'} = \{ p \in \mathcal C_\xi \ |\ \langle \xi', p \rangle = 1\}, $$
$$ \mathcal P_{\xi',\alpha} = \mathcal C_{\xi, \alpha} \cap P_{\xi'}, $$
$$ \Delta_{\xi',\alpha} = \{ p \in \mathcal C_{\xi, \alpha} \ |\ \langle \xi', p \rangle \le 1\}.$$
We now consider the functional 
$W :  \Xi_\bfo \to \bfR$ defined by
\begin{eqnarray}
 W(\xi') &:=& \sum_{\alpha=1}^k \log \frac{\Vol(\mathcal P_{\xi',\alpha})}{|\xi'|}\\
 &=& \sum_{\alpha=1}^k \log ((m+1) \Vol(\Delta_{\xi'})).
\end{eqnarray}

\begin{theorem}\label{thm1}
Let $S$ be a toric Sasaki manifold with Calabi-Yau condition of the K\"ahler cone,  i.e. 
$c_1^B(S) > 0$ and $c_1(D) = 0$. 
\begin{enumerate}
\item[(1)] $W$ is a strictly convex function on $\Xi_\bfo$.
\item[(2)] If we have a critical point $\xi' \in \Xi_{\bfo}$ such that
$$\mathcal P_{\xi'} = \mathcal P_{\xi',1}+\ \cdots +\ \mathcal P_{\xi',k}$$
then there exist transverse couple K\"ahler-Einstein metrics with respect to $\xi'$.
\item[(3)] In the case of $k=1$, if we take $\gamma_1 = c_1^B(S)$ and $\mathcal P_{\xi,1} = \mathcal P_\xi$ then
we have $\mathcal P_{\xi',1} = \mathcal P_{\xi'}$ for any $\xi' \in \Xi_\bfo$, and thus the
assumption in (2) is always satisfied. Further, the functional $W$ is a strictly convex
proper function and always have a critical point.
\end{enumerate}
\end{theorem}
The part (3) is due to Martelli-Sparks-Yau \cite{MSY1, MSY2}, and the part (2) is an attempt
to extend their argument to the coupled case. However, even if we assume 
$ c_1^B(S) = (\gamma_1 + \cdots + \gamma_k)/2\pi$
and 
$\mathcal P_\xi = \mathcal P_{\xi,1} + \cdots + \mathcal P_{\xi,k}$
it is not clear whether
$\mathcal P_\xi' = \mathcal P_{\xi',1} + \cdots + \mathcal P_{\xi',k}$
for other $\xi' \in \Xi_\bfo$, and can not conclude the existence of transverse coupled 
K\"ahler-Einstein metrics. In the last section, this will be explained using the $CR$ $f$-twist
of Apostolov-Calderbank \cite{ApoCal21MA}.

The volume minimization arguments were used for the studies of non-linear problems extending
K\"ahler-Einstein metrics in which Killing vector fields or Killing potentials are involved:
K\"ahler-Ricci solitons \cite{TianZhu02}, Sasaki-Einstein metrics \cite{MSY1, MSY2} and
conformally K\"ahler, Einstein-Maxwell metrics \cite{FO17}, see also the survey \cite{FO_ICCM_Notices19}.
In all these cases, a volume functional is defined on the space of Killing vector fields, and
the derivative is an obstruction to the existence of those metrics which extends the obstruction
to the existence of K\"ahler-Einstein metrics \cite{futaki83.1}.

After this introduction, the plan of this paper is as follows. In section 2 we review basic facts about
Sasaki manifolds. In section 3 we review known facts
on transverse K\"ahler geometry and the transverse moment map. In section 4, we map a Minkowski
sum decomposition of the image of the transverse
moment map to the image of the contact moment map to obtain a Minkowski sum decomposition of
of the contact moment map image. The properties of the latter is stated as Proposition \ref{thm2}.
Then we prove Theorem \ref{thm3}.
In section 5 we use the volume minimization argument and prove Theorem \ref{thm1}. 

\section{Deformations of Sasakian structures}
Let $S$ be a $(2m+1)$-dimensional smooth manifold. A 
 {\it contact structure} on $S$ is a $2m$-dimensional distribution $D \subset TS$
 such that the {\it Levi form} $L_D : D \times D \to TS/D$ defined by
 $$ L_D(X,Y) = - \eta_D([X,Y])$$
 is non-degenerate where $\eta_D : TS \to TS/D$ is the projection. The pair $(S,D)$ is called a {\it contact
 manifold}, and $D$ is also called the {\it contact distribution}. We assume $TS/D$ is an oriented real
 line bundle. If $\tau$ is a positive section of $TS/D$, then $\eta_\tau = \tau^{-1} \eta_D$ is a contact
 form, i.e. $d\eta_\tau|_D$ is non-degenerate. Then there is a unique vector field $\xi$, called the 
 {\it Reeb vector field}, such that
 $$ i(\xi) \eta_\tau = 1, \qquad\qquad i(\xi) d\eta_\tau = 0 $$
 where $i(\xi)$ denotes the inner product by $\xi$. 
 In this case $\eta_D(\xi) = \tau$. The flow on $S$ generated by $\xi$, i.e. the one parameter group
 of transformations generated by $\xi$, is called the {\it Reeb flow}. 
 Since $i(\xi) d\eta_\tau = 0$ and $L_\xi d\eta_\tau = 0$ where $L_\xi$ denotes the Lie derivative 
 by $\xi$, then $d\eta_\tau$ descends to a symplectic form on local orbit spaces of the Reeb flow.
 Let $\Omega^k_B(S)$ denote the set of all $k$-forms on $S$ which are obtained by pulling back from the local
 orbit spaces of the Reeb flow. We call such forms {\it basic $k$-forms} with respect to $\xi$. Obviously
 a $k$-form $\alpha$ on $S$ belongs to
 $\Omega_B^k(S)$ if and only if $i(\xi)\alpha = 0$ and $L_\xi \alpha = 0$. 
 The $2$-form $d\eta_\tau$ is a typical example of a basic $2$-form.
 
 A vector field $X$ on $S$ is said to be a {\it contact vector field} if $L_X C^\infty(D) \subset
 C^\infty(D)$. 
 \begin{lemma}\label{contact1}
 Sending a contact vector field $X$ to $\eta_D(X) \in C^\infty(TS/D)$ gives an isomorphism
 between the Lie algebra of contact vector fields and $C^\infty(TS/D)$.
 If $\sigma = f \tau \in C^\infty(TS/D)$ for a smooth function $f \in C^\infty(S)$, then 
 the corresponding contact vector field $X$ with $\eta_D(X) = \sigma$ is expressed as 
 \begin{equation}\label{contact2}
  X = f\xi + K_f
   \end{equation}
 for $K_f \in C^\infty(D)$ satisfying
 \begin{equation}\label{contact3}
 i(K_f) d\eta_\tau|_D = - df|_D
 \end{equation}
 \end{lemma}
 \begin{proof} We only show \eqref{contact3}. 
 Other part of the proof is left to the reader, or see the proof of Lemma 1 in \cite{ApoCal21MA}, p.1055.
 If $X$ is given by \eqref{contact2} then $\eta_D(X) = f\tau = \sigma$. 
 For $Y \in C^\infty(D)$ we have
 $$ (i(X) d\eta_\tau)(Y) = -Yf$$
 since $L_X C^\infty(D) = C^\infty(D)$. This implies
 $$ (i(X) d\eta_\tau)|D = -df|_D.$$
 On the other hand, using \eqref{contact2} we have 
 $$ i(X) d\eta_\tau = i(K_f) d\eta_\tau.$$
 Then \eqref{contact3} follows from the last two equations.
 \end{proof}
 
 \begin{lemma}\label{contact4}
 In Lemma \ref{contact1}, if $[X,\xi] = 0$ then $f$ is a basic function with respect to $\xi$, i.e. $\xi f = 0$,
 and $K_f$ descends to a Hamiltonian vector field of $f$ on local orbit spaces of the Reeb flow.
 \end{lemma}
 \begin{proof} One can show that $f$ is basic by 
 $$ 0 = (d\eta_\tau) (X,\xi) = X\eta_\tau(\xi) - \xi (\eta_\tau(X)) - \eta_\tau([X,\xi]).$$
 One can also show $[K_f,\xi]=0$ using \eqref{contact2} together with $[X,\xi]=0$ and $\xi f = 0$.
 Thus $K_f$ descends to the local orbit spaces of the Reeb flow, and \eqref{contact3} shows that $K_f$ is the
 Hamiltonian vector field of $f$.
 \end{proof}
 
 Let $J \in \mathrm{End}(D)$ be an almost complex structure of the contact distribution 
 $D$, i.e. $J^2 = - \mathrm{id}$. We say that $(D,J)$ is a $CR$-structure if
 $$ D^{1,0}: = \{X - iJX\, |\, X \in D\}$$
 is involutive, i.e. the set $C^\infty(D^{1,0})$ of smooth sections of $D^{1,0}$ is closed under the
 bracket.  A $CR$-structure $(D,J)$ is said to be strictly pseudo-convex if $d\eta_\tau(\cdot\,,\, J\cdot)|_D$
 is a positive definite Hermitian form for a positive section $\tau$. Note that this definition is independent of
 the choice of a positive section $\tau$. Then the triple $(S,D,J)$ is called a strictly pseudo-convex 
 $CR$-manifold.
 
 A contact vector field $\xi$ is said to be a {\it $CR$-vector field} if $L_\xi J = 0$.
 
 \begin{definition}\label{Sasaki1}
 If $\xi$ is a $CR$-vector field on a strictly pseudo-convex $CR$-manifold and $\eta_D(\xi)$ gives a
 positive section then we call $(D,J,\xi)$ a Sasaki structure and $(S,D,J,\xi)$ a Sasaki manifold. 
 The $CR$-vector field $\xi$ is called the {\it Reeb vector field} of the Sasaki manifold.
 (This definition of Reeb vector field is compatible with the above definition if we take 
 $\tau = \eta_D(\xi)$.)
 \end{definition}
 
 Since $L_\xi J = 0$ and $D^{1,0}$ is involutive then the local orbit spaces of the Reeb flow have
 a complex structure. Further, for the submersions $\pi_i : U_i \to V_i$ of a small open set $U_i \subset S$
 onto a local orbit space $V_i$, the map
 $$ \pi_i\circ\pi_j^{-1} |_{\pi_j(U_i \cap U_j)} : \pi_j(U_i \cap U_j) \to \pi_i(U_i \cap U_j) $$
 is biholomorphic. The collection of such $(U_i, V_i, \pi_i)$'s is called a {\it transverse holomorphic structure}.
 Further, by the property of strong pseudo-convexity, $\frac12 d \eta_\tau$ descends to $V_i$'s to define 
 K\"ahler forms $\omega_i$'s, and $\pi_i\circ\pi_j^{-1} |_{\pi_j(U_i \cap U_j)}$'s are K\"ahler isometries. We call
 the collection of such $(U_i, V_i, \pi_i, \omega_i)$'s a {\it transverse K\"ahler strucure}.
 
 \begin{remark}\label{rem1}
 The convention of $\frac12 d \eta_\tau$, but not $d \eta_\tau$, is the standard choice of the transverse
 K\"ahler form. This makes $\frac12 i\partial\barpartial r^2$ the K\"ahler form on the cone $C(S)$, see
 the proof of Proposition \ref{Sasaki8}.
 \end{remark}
 
 \begin{lemma}\label{Sasaki2}
 Let $(S, D, J)$ be a strictly pseudo-convex $CR$-manifold. Suppose that $\xi$ and $\xi'$ be two 
 commuting Reeb
 vector fields, i.e. $[\xi,\xi'] = 0$, giving rise to two Sasaki structures $(D,J,\xi)$ and $(D,J,\xi')$ on $S$.
Then $\xi'$ is expressed as 
\begin{equation}\label{Sasaki3}
 \xi' = f\xi + K_f
\end{equation}
where $K_f \in C^\infty(D)$ descends to a Killing vector field on local orbit spaces of the Reeb flow of $\xi$
and $f$ is a positive basic function with respect to $\xi$ which descends to a Killing potential of $K_f$. 
 \end{lemma}
 \begin{proof}
 A section $Y \in C^\infty(D)$ such that $L_\xi Y = 0$ descends to a 
 vector field on local orbit spaces of the Reeb flow of $\xi$, which we denote by $Y^\vee$.
 For another $W \in C^\infty(D)$ such that $L_\xi W = 0$ we have
 $$ [Y^\vee, W^\vee] = ([Y, W] - \eta_\tau ([Y,W])\xi)^\vee.$$
 By Lemma \ref{contact4} $K_f^\vee$ is a Hamiltonian vector field of the basic function $f$.
 To show that $K_f^\vee$ is a Killing vector field we need to show $L_{K_f^\vee}J^\vee = 0$.
 This is equivalent to
 \begin{equation}
 [K_f^\vee, J^\vee Y^\vee] = J^\vee[K_f^\vee, Y^\vee].
  \end{equation}
 But this follows from the following three equalities:
 $$ ([\xi',JY] - \eta_\tau([\xi',JY]\xi)^\vee = J^\vee[\xi',Y]^\vee,$$
 \begin{eqnarray*}
([f\xi,JY] - \eta_\tau ([f\xi,JY])\xi)^\vee &=& f[\xi,JY] - ((JY)f)\xi - \eta_\tau(f[\xi,JY] - ((JY)f)\xi)\xi\\
&=& (fJ[\xi,Y] - ((JY)f)\xi  + \eta_\tau(((JY)f)\xi)\xi)^\vee\\
&=& (fJ[\xi,Y])^\vee
 \end{eqnarray*}
 and
\begin{eqnarray*}
J^\vee[f\xi,Y]^\vee &=& (J([f\xi,Y] - \eta_\tau([f\xi,Y])\xi))^\vee \\
&=& (J( f[\xi,Y] - (Yf)\xi - \eta_\tau(f[\xi,Y] - (Yf)\xi)\xi)^\vee\\
&=& (fJ[\xi,Y])^\vee.
\end{eqnarray*}
Alternatively, one may simply argue that, since the flow generated by $\xi'$ preserves $J$
and this flow descends to the flow generated by $K_f^\vee$ preserving $J^\vee$,
we obtain $L_{K_f^\vee}J^\vee = 0$. 

 Since both $\eta_\tau(\xi)$ and $\eta_\tau(\xi')$ give positive orientation then
 $$ f = \frac{\eta_\tau(\xi')}{\eta_\tau(\xi)} > 0.$$
 This completes the proof of Lemma \ref{Sasaki2}.
 \end{proof}
 
 Associated with a Sasaki structure we have a Riemannian metric $g$ defined by
 $$ g(\xi,\xi) = 1, \qquad\qquad g(\xi, D) = 0,$$
 and
 \begin{equation}\label{Sasaki4}
 g_D = \frac12 d\eta_\tau (\cdot\,\, J\cdot).
 \end{equation}
 The Riemannian manifold $(S,g)$ associated with the Sasaki Structure $(S,D,J,\xi)$ as above
 is often called a Sasaki manifold. The normalization of $g(\xi,\xi) = 1$ is the standard choice,
 and this choice determines various constants. For example, if $(S,g)$ is an Einstein manifold,
 called a {\it Sasaki-Einstein manifold}, 
 then the Ricci curvature $\mathrm{Ricci}$ satisfies
  \begin{equation}\label{Sasaki5}
 \mathrm{Ricci} = 2m\, g,
 \end{equation}
 and in this case the transverse K\"ahler metric is K\"ahler-Einstein with
 the transverse Ricci curvature $\mathrm{Ricci}^T$ satisfying
  \begin{equation}\label{Sasaki6}
 \mathrm{Ricci}^T = (2m+2)\, g^T
 \end{equation}
 where $g^T$ is the K\"ahler metric on the local orbit spaces of the Reeb flow
 naturally induced from $g_D$, called the transverse K\"ahler metric.
 Note that many authors use different conventions of Ricci curvature
 between Riemannian geometry and K\"ahler geometry
 because the trace is taken respect to an orthonormal basis of the real tangent bundle 
 in Riemannian geometry while holomorphic tangent bundle 
 in K\"ahler geometry, resulting in
K\"ahlerian Ricci curvature being a half of the Riemannian Ricci curvature.
 Ricci curvature in \eqref{Sasaki6} is Riemannian Ricci curvature.
 To distinguish them we denote the K\"ahlerian Ricci curvature by $\Ric$
 while the Riemannian Ricci curvature by $\mathrm{Ricci}$.
 Thus \eqref{Sasaki6} is equivalent to
 \begin{equation}\label{Sasaki7}
 \Ric^T = (m+1)\, g^T.
 \end{equation}
 By the same reason the K\"ahlerian scalar curvature is one fourth of Riemannian
 scalar curvature. In this paper we only deal with the K\"ahlerian transverse scalar
 curvature, which we denote by $R^T$. Thus, for a Sasaki-Einstein manifold,
 we have
 \begin{equation}\label{Sasaki8}
 R^T = m(m+1).
 \end{equation}

 Note also that the associated {\it transverse K\"ahler form} $\omega^T$ is given by
 \begin{equation}\label{Sasaki9}
  \omega^T = \frac12 d\eta_\tau.
 \end{equation}
 As the expression of \eqref{Sasaki9} shows, the transverse K\"ahler form $\omega^T$ is identified with 
 the differential 2-form $\frac12 d\eta_\tau$ on $S$ as a basic 2-form.
 
 Next we see that a Sasaki manifold is obtained as a link of a K\"ahler cone.
 \begin{definition}
 Let $(S,g)$ be a Riemannian manifold. The Riemannian cone $(V, g_V)$
 of $(S,g)$ is the pair of the product manifold $V = \bfR_+ \times S$ and the
 warped product metric
 $$ g_V = dr^2 + r^2g$$
on $V$ where $r$ is the standard coordinate of $\bfR_+$.
 \end{definition} 
 Let $(V,g_V)$ be a Riemannian cone of $(S,g)$ as above. If $(V,g_V)$ is K\"ahler then $(S,g)$ is
 a Sasaki manifold in the following way. We identify $(S,g)$ as the submanifold $\{r=1\}$ in
 $(V,g_V)$. Then 
 $$\eta = d^cr |_{S = \{r=1\}}$$ 
 is a contact form
 where our convention of $d^c$ is
 $$(d^c f)(Y) = df(-JY) = (i(\barpartial - \partial)f)(Y),$$
 $D = \ker \eta$ is a contact distribution, and $\xi := Jr\frac{\partial}{\partial r}|_{r=1}$ is the
Reeb vector field.
 
 \begin{proposition}\label{Sasaki10}
 Let $(S,D,J,\xi)$ be a Sasaki manifold as defined in Definition \ref{Sasaki1}. Then there is a K\"ahler
 cone $(V,g_V)$ such that the Sasakian structure on $\{r=1\} \subset V$ described as above is
 isomorphic to $(S,D,J,\xi)$.
 \end{proposition}
 \begin{proof}
 Let $(V,g_V)$ be the Riemannian cone of $(S,g)$. We show that the almost complex structure $J$ on
 the $CR$-structure $D$ extends to an integrable complex structure $J$ on $V$ such that $(g_V, J)$
 is K\"ahler. Extend $J$ on $D$ to $TV$ by
 $$ J(r\frac{\partial}{\partial r}) = \xi.$$
Consider $d^cr = dr(-J\cdot)$ with respect to $J$ on $V$. Then one can check 
$$ d^cr |_{\{r=1\}} = \eta_\tau, \qquad \tau = \eta_D(\xi).$$
We extend $\eta_\tau$ to $V$ by
$$ \eta_\tau = d^c \log r. $$
One can then show that $dr + ir\eta_\tau$ and $C^\infty(D^{1,0\, \ast})$ generates sections of the type $(1,0)$
cotangent bundle of $V$ and that they form a differential ideal. Thus $J$ on $V$ is integrable.

Then since $g_V = dr^2 + r^2 g_S$ its fundamental 2-form $\omega_V$ is computed by
\begin{eqnarray*}
\omega_V &=& dr(J\cdot)\wedge dr + r^2 g_S(J,\cdot, \cdot)\\
&=& dr \wedge d^c r +\frac{r^2}2 d\eta_\tau\\
&=& \frac12 \partial\barpartial r^2.
\end{eqnarray*}
This shows that $g_V$ is a K\"ahler metric.
 \end{proof}
 Thus we have an equivalent definition of Sasaki manifolds:
 \begin{definition}\label{Sasaki11}
 An odd dimensional Riemannian manifold $(S,g)$ is called a Sasaki manifold
 if its Riemannain cone $(V,g_V)$ is K\"ahler.
 \end{definition}
 
\section{Deformations of the transverse K\"ahler structures}

As described in the previous section, a Sasakian structure on a differentialble manifold $S$ is
given by the triple of $(D,J,\xi)$ where $(D,J)$ is a strongly pseudo-convex $CR$-structure and
$\xi$ is a Reeb vector field. When we consider deformations of Sasakian structures we may
separate into two types of deformations, one fixing $(D,J)$ and the other fixing $\xi$.
The case fixing $(D,J)$ was already considered in Lemma \ref{Sasaki2}, in which $\xi$ and $\xi'$
are commutative. This leads us to the setting where a compact real torus $T$ of dimension $n$ 
acts effectively on a 
Sasaki manifold $(S,D,J,\xi)$ in such a manner that $T$ preserves $(D,J)$ and the Lie algebra $\mathfrak t$ of $T$ contains $\xi$.
 Then $T$-action extends naturally as isometries of the K\"ahler cone associated to $(S,D,J,\xi)$,
 and the image of its moment map $\mu_\xi : V \to \mathfrak t^{\ast}$ is a convex 
 polyhedral cone (\cite{Lerman02Illinois}). If further $n=m+1$, which is the maximal dimension
 of the effective torus action, $V$ is a toric manifold and the complex structure of $V$ is
 invariant under $T$-invariant deformations of $(D,J,\xi)$. 
 
 Motivated by this toric setting, when we consider deformations fixing $\xi$ we restrict ourselves 
 to the situation where the complex structure of the K\"ahler cone $V$ is fixed.
 Then the transverse holomorphic structure on the local orbit spaces of the flow generated by $\xi$ is also fixed.
Thus we try to deform the transverse K\"ahler form
$$ \omega^T = \frac12 d\eta_\xi$$ 
into another transverse K\"ahler form with different contact structure
where we have written $\eta_\xi$ instead of $\eta_\tau$ with $\tau = \eta_D(\xi)$ and keep this new
 notation hereafter. As in K\"ahler geometry one may deform the transverse K\"ahler form
 $\omega^T$ into
 $$ \omega^T_\varphi = \omega^T + i \partial_B\barpartial_B\varphi$$
 using a basic smooth function $\varphi \in C^\infty_B(S)$. Here $\partial_B$ and $\barpartial_B$ 
 are basic $\partial$ and $\barpartial$ operators which are naturally defined on complex valued
 basic forms, and $\barpartial_B$ naturally defines basic Dolbeault cohomology $H_B^{p,q}(S)$.
 Then we have
 \begin{eqnarray*}
  \omega^T_\varphi &=& \frac12 dd^c \log re^{2\varphi}\\
  &=& \frac12 d (\eta_\xi + 2d^c \varphi).
 \end{eqnarray*}
Thus 
\begin{equation}\label{SE0}
D':= \ker (\eta_\xi + 2d^c \varphi)
\end{equation}
is a variation of $D$ with fixed Reeb vector field $\xi$.
This deformation of $D$ into $D'$ is regarded as a deformation of the cone manifold structure
by changing the radial function $r$ into $re^{2\varphi}$. As argued in \cite{HeSun16Adv, CollinsSzeke18JDG, 
BoyerCoev18MRL, ApoCalLeg21Adv}, if we fix $D$ and deform $\xi$ we can take $\{r'=1\} =
\{r = 1\} = S$, but when $\xi$ is fixed and $D$ is deformed the submanifold $\{r = 1\}$ has to change
into $\{ r e^{2\varphi} = 1\}$.

Let $\omega$ be arbitrary basic K\"ahler form (not necessarily equal to $ \omega^T = \frac12 d\eta_\xi$). 
Denote by $\rho^T(\omega)$ the transverse Ricci form associated with the transverse Ricci curvature $\Ric^T(\omega)$:
 $$ \rho^T(\omega) = -i \partial_B\barpartial_B \log \omega^m.$$
 Then the basic cohomology class represented by $\rho^T(\omega)/2\pi$ is independent of the choice of
 the K\"ahler form $\omega$. We call this basic cohomology class the {\it basic first Chern class}
 and denote it by $c_1^B(S)$.

Next we recall known results about an obstruction to the existence of transverse K\"ahler-Einstein metrics 
and transverse coupled K\"ahler-Einstein metrics
studied in \cite{futaki83.1}, \cite{futaki87}, \cite{futaki88}, \cite{FOW}, \cite{futaki10}, \cite{FutakiZhang18}, \cite{FutakiZhang19}, \cite{futaki21PAMQ}.
Suppose $c_1^B(S) > 0$ and choose this basic class as a basic K\"ahler class.
Unless we assume $c_1(D) = 0$ as in Lemma \ref{SE2} in section \ref{section4},
this class may not be a positive multiple of the standard transverse K\"ahler class $[\omega^T]_B = [\frac12 d\eta_\xi]_B$ where $[\cdot]_B$ denotes the basic cohomology class,
but for later applications we have in mind the case when $c_1^B(S) = (m+1)[\frac12 d\eta_\xi]_B$.

Let $T$ be a real compact torus acting effectively as $CR$-automorphisms
of $(D,J)$ such that $\xi$ is contained in the Lie algebra $\mathfrak t$ of $T$.
We identify a Lie algebra element $X \in \mathfrak t$ as a smooth vector field on $S$.
Further, by Lemma \ref{Sasaki2}, $X$ descends to a holomorphic Killing potential on
local orbit spaces of the Reeb flow of $\xi$, 
thus as a transverse holomorphic vector field. We take this view point below.
Let $\omega$ be a $T$-invariant basic K\"ahler form in $\frac1{m+1} c_1^B(S)$.
Let $F$ be a $T$-invariant basic smooth function on $S$ such that
\begin{equation}\label{Fut0}
 \rho^T(\omega) = (m+1)\omega + i\partial_B\barpartial_B F.
 \end{equation}
Then since $c_1^B(S) > 0$, for any $X \in \mathfrak t$ there exists a smooth basic
function $v$ such that
\begin{equation}\label{Fut0.1}
 i(X)\omega = - dv.
 \end{equation}
 Note that for $\xi \in \mathfrak t$, $i(\xi)\omega = 0$ since $\omega$ is basic and thus $v=0$.
Then with the normalization of $v$, by
\begin{equation}\label{Fut0.2}
 \int_S v\,e^F \omega^m\wedge \eta_\xi = 0
  \end{equation}
the same arguments as in Proposition 4.1 in \cite{futaki87} (see also Theorem 2.4.3 in \cite{futaki88})
one can show that $v$ satisfies
\begin{equation}\label{Fut1}
\Delta_B v + v^i F_i + (m+1)v = 0
\end{equation}
where $\Delta_B$ denotes the $\barpartial_B$-Laplacian.
Note that 
\begin{eqnarray*}
v^i F_i &=& g^{i\barj}v_{\barj}F_i\\
&=&\frac i2 (X -iJX)F \\
&=& \frac12 (JX)F
\end{eqnarray*}
since $F$ is $T$-invariant. Define $\Fut : \mathfrak t/\bfR \xi \to \bfR$ by
\begin{eqnarray}
 \Fut(X) &=& \int_S v^i F_i\, \omega^m\wedge\eta_\xi \label{Fut2}\\
 &=& \int_S \frac12 (JX)F\,\omega^m\wedge\eta_\xi. \nonumber
\end{eqnarray}
Then as in \cite{futaki83.1}, \cite{FOW}, $\Fut$ is independent of choice of $\omega$ in 
$\frac1{m+1} c_1^B(S)$, and the non-vanishing of $\Fut$ obstructs the existence of
a transverse K\"ahler-Einstein metric in $\frac1{m+1} c_1^B(S)$ by \eqref{Fut0}.
This invariant can be expressed in terms of the transverse moment map
$\mu^T : S \to (\mathfrak t/\bfR \xi)^\ast$
\begin{equation}\label{Fut3}
\langle \mu^T(x) , X \rangle = v(x)
\end{equation}
where $v$ is related with $X$ by \eqref{Fut0.1} with the normalization \eqref{Fut0.2}.
The image of $\mu^T$ is a compact convex polytope, and this polytope is unchanged
even if the K\"ahler form $\omega$ is changed in the same cohomology class.
This can be checked by noting that the vertices of the polytopes are the critical values
of $v$'s and that if $\omega$ changes to 
$\omega_\varphi = \omega + i\partial_B\barpartial_B\varphi$ then $v$ changes to
$v + v^\alpha \varphi_\alpha$, and the critical values do not change. Notice that
\eqref{Fut0.2} and \eqref{Fut1} are also preserved under these changes.
Then since
\begin{equation}\label{Fut4}
\Fut(X) = - (m+1)\int_S v\, \omega^m \wedge \eta_\xi
\end{equation}
by \eqref{Fut1} it follows that $\Fut$ vanishes if and only if the barycenter of the 
image of the moment map $\mu^T$ lies at zero. 

To express this moment polytope of $\mu^T$ we let $K_S$ denote the complex
line bundle over $S$ consisting of basic $(m,0)$-forms. Then $c_1^B(S) = c_1(K_S^{-1})$.
The moment polytope of $\mu^T$ is associated with the basic K\"ahler class $\frac1{m+1} K_S^{-1}$.
Thus we express the moment polytope of $\mu^T$ by $\frac1{m+1}\mathcal P_{-K_S}$.

Next we recall an obstruction to the existence of transverse coupled K\"ahler-Einstein
metrics. Suppose $c_1^B(S) > 0$. 
A {\it decomposition} of $c_1^B(S)$ is a sum
\begin{equation*}
 c_1^B(S) = (\gamma_1 + \cdots + \gamma_k)/2\pi
\end{equation*}
of positive basic $(1,1)$ classes $\gamma_\alpha/2\pi$. 
If we choose basic K\"ahler forms 
$\omega_\alpha$ representing $\gamma_\alpha$, there exist smooth basic functions $F_\alpha$ such that 
\begin{equation}\label{potential}
\rho^T(\omega_\alpha) - \sqrt{-1}\partial_B\barpartial_B F_\alpha = \sum_{\beta=1}^k \omega_\beta, \quad \alpha = 1,\ \cdots,\ k.
\end{equation}
We say $\omega_\alpha$'s are {\it transverse
coupled K\"ahler-Einstein metrics} if $F_\alpha$ is constant
so that transverse coupled K\"ahler-Einstein metrics satisfy 
\begin{equation*}
\rho^T(\omega_1) = \cdots = \rho^T(\omega_k) = \sum_{\beta=1}^k \omega_\beta.
\end{equation*}
We will also call transverse
coupled K\"ahler-Einstein metrics {\it coupled Sasaki-Einstein metrics} when we further
assume $c_1(D) = 0$ since, for $k=1$, $\frac1{m+1}\omega_1$ is a Sasaki-Einstein metric
with respect to a modified contact form as in \eqref{SE0}. 

We choose $\omega_\alpha$ in 
$\gamma_\alpha$ and normalize $F_\alpha$ so that 
\begin{equation}\label{coupledS2}
e^{F_1}\omega_1^m= \cdots = e^{F_k}\omega_k^m
\end{equation}
and put
\begin{equation}\label{coupledS3}
dV := e^{F_\alpha} \omega_\alpha^m \wedge \eta_\xi.
\end{equation}
For $X \in \mathfrak t$, let $v_\alpha$ be the basic function satisfying
\begin{equation}\label{coupledS3.1}
 i(X)\omega_\alpha = - dv_\alpha
 \end{equation}
 with normalization condition 
 \begin{equation}\label{coupledS4}
 \int_S (v_1 + \cdots + v_k)\ dV = 0.
 \end{equation}
 Then as proved in \cite{FutakiZhang18}, Theorem 3.3, $v_\alpha$'s satisfy
 \begin{enumerate}
\item $\nabla_\alpha^i v_\alpha=\nabla_\beta^i v_\beta$ for  $i=1, 2, \dots, n$ and $\alpha, \beta=1, 2, \dots, k$.
\item $\Delta_\alpha v_\alpha+v^i _\alpha F_{\alpha i}+ \sum\limits_{\beta=1}^k v_\beta = 0$ for $\alpha=1, 2, \dots, k$, where $\Delta_\alpha=-\barpartial_\alpha^*\barpartial$ is the Laplacian with respect to the K\"ahler form 
$\omega_\alpha$.
\end{enumerate}
It is also shown in \cite{FutakiZhang18}, Theorem 5.2, that \eqref{coupledS4} is equivalent to
the Minkowski sum relation
$$ \sum_{\alpha=1}^k \mathcal P_\alpha =\mathcal P_{-K_S}$$
where $\mathcal P_\alpha$ is the moment polytope for $\omega_\alpha$.
This of course means
\begin{equation}\label{coupledS5}
\sum_{\alpha=1}^k \frac1{m+1}\mathcal P_\alpha =\frac1{m+1}\mathcal P_{-K_S},
\end{equation}
the right hand side being the moment polytope for $\mu^T$. 
Define $\Fut^{cpld} : \mathfrak t \to \bfR$ by
\begin{equation}\label{cpldfutaki}
\Fut^{cpld}(X)
= \sum\limits_{\alpha=1}^k\frac{\int_S v_\alpha\ \omega_\alpha^m\wedge\eta_\xi}{\int_S \omega_\alpha^m\wedge\eta_\xi}.
\end{equation}
Then $\Fut^{cpld}$ is independent of the choice of $\omega_\alpha$ in $\gamma_\alpha$,
the nonvanishing of $\Fut^{cpld}$ obstructs the existence of transverse coupled K\"ahler-Einstein metrics
and $\Fut^{cpld}$ vanishes if and only if the sum of the barycenters of $\mathcal P_\alpha$ lies
at zero (\cite{FutakiZhang18}, Theorem 1.4). In the toric case where $\dim T = m+1$, the vanishing
of $\Fut^{cpld}$ is a sufficient condition for the existence of transverse coupled K\"ahler-Einstein
metrics, which is a straightforward extension of the celebrated results of Wang-Zhu \cite{Wang-Zhu04}, Donaldson
\cite{donaldson0803}, Hultgren \cite{Hultgren17}.

As a summary of this section, the obstructions $\Fut$, $\Fut^{cpld}$ and the transverse 
moment polytopes of $\mu^T$ depend only
on the Reeb vector field $\xi$, its basic first Chern class $c_1^B(S)$ and its decomposition.
But as mentioned above we have in mind the case $c_1^B(S) = (m+1)[\frac12 d\eta_\xi]_B$,
in which case $\Fut$, $\Fut^{cpld}$ are independent of the choice of $D'$ of the form \eqref{SE0}.
As a conclusion of this section, when we study $\Fut$, $\Fut^{cpld}$ and moment polytopes under
the variation of Reeb vector fields we may choose an arbitrary strongly pseudo-convex 
$CR$-structure $(D,J)$, fix it, and thus are in the position of Lemma \ref{Sasaki2}.

\section{The transverse moment map and the contact moment map}\label{section4}

The equality \eqref{Sasaki7} shows that a
 necessary condition for the existence of a Sasaki-Einstein metric is
  \begin{equation}\label{SE1}
 c_1^B(S) = (m+1) [\frac12 d\eta_\xi]_B
 \end{equation}
 as a basic cohomology. 
Here $[d\eta_\xi]_B$ denotes the basic cohomology class which is a positive $(1,1)$-class as a
  basic class though it is zero as a de Rham class of $S$. 
Recall that the basic first Chern class is said to be positive, denoted $c_1^B(S) > 0$, if it is represented by a positive basic $(1,1)$-form, i.e. a basic K\"ahler form. This is an obvious necessary condition for the existence of Sasaki-Einstein metric by \eqref{SE1}.
The following lemma is well-known and important for us, see \cite{BGbook}.  
\begin{lemma}\label{SE2}
If $c_1^B(S) > 0$ and $c_1(D) = 0$ then by changing $r$ into $r^a$ for a positive constant $a$ if necessary we can assume
\eqref{SE1} is satisfied. (The transformation from $r$ into $r^a$ is called the $D$-homothetic transformation.)
\end{lemma}
\begin{proof}
If $c_1(D) = 0$, then $c_1^B = \tau [d\eta_\xi]$ for some 
constant $\tau$ 
by \cite{BGbook}, Corollary 7.5.26. Since $c_1^B(S) > 0$ we must have $\tau > 0$. Then we may take $a = (m+1)/\tau$. 
\end{proof}

The summary at the end of previous section 
is an observation concerning the transverse moment map $\mu^T$, but we may compare
it with the contact moment map $\mu^{con} : V \to \mathfrak t^\ast$ defined by 
\begin{equation}\label{con}
\langle \mu^{con}(x), X \rangle = (r^2\eta_\xi(X))(x).
\end{equation}
The linkage between the transverse moment map $\mu^T$, which is defined with respect to the
transverse K\"ahler class $c_1^B(S)$, and the contact moment map $\mu^{con}$, which is defined
with respect to the contact form $\eta_\xi$, 
is
played by the conditions $c_1(D) = 0$ and $c_1^B(S) > 0$ in 
Lemma \ref{SE2}. The relation of $v$ in \eqref{Fut3} and $\eta_\xi(X)$ in \eqref{con} is that
$$ v = \frac{m+1}2 \eta_\xi(X) + c$$
where $c$ is determined by the normalization \eqref{Fut0.2}.

By \cite{Lerman02Illinois}, the image of $\mu^{con}$ is a convex polyhedral cone, which we denote 
by $\mathcal C$. Identifying $S$ with $\{r=1\}$ we have the moment map of $S$ by
restricting $\mu^{con}$ to $\{r=1\}$. The image of $S$ is 
$$ \mathcal P_\xi := \mathrm{Image}(\mu^{con}) \cap \{p \in \mathfrak t^\ast\,|\,\langle p, \xi\rangle = 1\}$$
since $\eta_\xi(\xi) = 1$. This set is called the {\it characteristic hyperplane} in $\mathcal C$.

Since the Hamiltonian functions for the basis of $\mathfrak t/\bfR \xi$ determine affine coordinates
on the images of $\mu^T$ and $\mu^{con}$, the map 
\begin{equation}\label{con2}
\Phi := \mu^T\circ(\mu^{con})^{-1}|_{\mathcal P_\xi} : \mathcal P_\xi
\to \frac1{m+1}\mathcal P_{-K_S}
\end{equation}
is an affine map in terms of those affine coordinates. Note that $\frac1{m+1}\mathcal P_{-K_S}$ is in
$(\mathfrak t/\bfR \xi)^\ast$ which is a vector space and contains the origin $0$ but that $ \mathcal P_\xi$
is in a hyperplane in the cone $\mathcal C \subset \mathfrak t^\ast$.

\begin{proof}[Proof of Proposition \ref{thm2}]
We take $\bfo$ and $\mathcal P_{\xi,\alpha}$ so that to be $\Phi(\bfo)=0$ and 
$\Phi(\mathcal P_{\xi,\alpha}) = \frac1{m+1}\mathcal P_\alpha$ in \eqref{coupledS5}. Then we obtain the
Minkowski decomposition as claimed in (1). Further, the barycenter of the domain polytope is mapped to the barycenter of the image polytope by an affine map, and thus the last claim in (1) follows. 

By \eqref{coupledS4}, $\mathcal P_\alpha$'s are unique up to translations $\mathcal P_\alpha + c_\alpha$
satisfying $c_1 + \cdots + c_k = 0$, from which (2) follows.

There is a unique cone $\mathcal C_{\xi, \alpha}$ such that 
$\mathcal C_{\xi, \alpha} \cap \mathcal P_\xi = \mathcal P_{\xi, \alpha}$. 
Then \eqref{Mink0} implies \eqref{Mink1}. This completes the proof of Proposition \ref{thm2}.
\end{proof}


Suppose now that the dimension of the torus $T$ is $m+1$ so that the K\"ahler cone $V = C(S)$ 
is a toric manifold. Then the moment map image is a convex polyhedral cone which is 
{\it good} in the sense of Lerman \cite{Lerman03JSG} described as follows.
Let $\mathcal C$  be the convex rational polyhedral cone described as
$$ \mathcal C = \{ p \in \mathfrak t^\ast\backslash \{o\}\,|\, \langle p, \ell_a\rangle \ge 0,\ a = 1, \cdots, d\}$$
where $\ell_a \in \mathfrak t$ such that $2\pi \ell_1, \cdots, 2\pi\ell_d$ are primitive elements of 
the kernel $\Lambda$ of $\exp : \mathfrak t \to T$. We regard $\ell_a$ as a linear function on $\mathfrak t^\ast$
and write
$$ \ell_a(p) = \langle p, \ell_a \rangle \ \ \text{for}\ \ p \in \mathfrak t^\ast.$$
We say $\mathcal C$ is good if for any 
face 
$F = \cap_{j=1}^k \{\ell_{a_j} = 0\}$ of codimension $k$ we have
\begin{equation}\label{good}
(\bfR \ell_{a_1} + \cdots + \bfR \ell_{a_k}) \cap \bfZ^{m+1} = \bfZ \ell_{a_1} + \cdots + \bfZ \ell_{a_k}.
\end{equation}
It is shown in \cite{Lerman03JSG} that if $S$ is smooth then $\mathcal C := \mu^{con}(C(S))$ is a good
convex rational polyhedral cone 
and conversely if $\mathcal C$ is a good
convex rational polyhedral cone 
then there is a smooth toric Sasaki manifold $S$ such that 
$\mu^{con}(C(S)) = \mathcal C$. 

For a toric Sasaki manifold there are descriptions using the action-angle coordinates 
\cite{Guillemin94JDG}, \cite{Abreu}, \cite{MSY1}, \cite{MSY2}. 
As explained in the Introduction, a salient fact
when $c_1^B(S) > 0$ and $c_1(D) = 0$ 
is that there is a distinguished point $\gamma \in \mathcal P_\xi$ with the property that
\begin{equation}\label{q1}
\langle \gamma, \xi \rangle = -(m+1)\qquad  \text{and}\qquad\ell_a (\gamma) = -1 
\quad  \text{for}\ a=1,\cdots, d,
\end{equation}
see \cite{MSY1}, \cite{CFO}.
Theorem \ref{thm3} claims that the origin $\bfo$ coincides with $q:=-\frac1{m+1}\gamma$.

\begin{proof}[Proof of Theorem \ref{thm3}]
By the proof of Proposition \ref{thm2}, 
we have only to show $\Phi(q) = 0$ for the affine map $\Phi$ defined \eqref{con2}.
This proof is motivated by the computations in \cite{deBorbonLeg20}. First of all, by Donaldson's expression of
the obstruction in \cite{donaldson02}
\begin{equation}\label{q2.1}
\Fut(X) =  \int_{\partial P_\xi} y \sigma_\xi -
\frac{\int_{\partial \mathcal P_\xi}\sigma_\xi}{\int_{\mathcal P_\xi}d\tilde x} \int_{\mathcal P_\xi} y\, d\tilde{x}
\end{equation}
where $y$ is an affine function corresponding to $X \in \mathfrak t$, $\tilde x = (\tilde x^1, \cdots, \tilde x^m)$
are affine coordinates on the hyperplane $\{p \in \mathfrak t^\ast\,|\, \langle p,\xi\rangle = 1\}$, and
$$ d\ell_a \wedge \sigma_\xi = - d\tilde x^1 \wedge \cdots \wedge d\tilde x^m \quad \text{on the facet}\ F_a \cap P_\xi.$$
Note that $\frac{\int_{\partial \mathcal P_\xi}\sigma_\xi}{\int_{\mathcal P_\xi}d\tilde x}$ is the average (K\"ahler 
geometers') scalar curvature, which is equal to $m(m+1)$ for $\omega \in 2\pi c_1^B(S)/(m+1)$.
As shown in \cite{deBorbonLeg20}, Lemma 3.8, $\sigma_\xi$ is expressed using the distinguished point $q$ by
\begin{equation}\label{q3}
\sigma_\xi = \frac1{m+1} \sum_{i=1}^m (-1)^{i+1}(\tilde x^i - q^i) d\tilde x^1 \wedge \cdots \wedge
\widehat{d\tilde x^i} \wedge \cdots \wedge d\tilde x^m.
\end{equation}
Using $d\sigma_\xi = m(m+1)d\tilde x$ and Stokes theorem, one can show
\begin{equation}\label{q4}
 \int_{\partial \mathcal P_\xi}  \sigma_\xi = m(m+1)\int_{\mathcal P_\xi} d\tilde x
\end{equation}
as expected to get average scalar curvature and 
\begin{eqnarray}
\int_{\mathcal P_\xi} \tilde x^i d\tilde x &=& \frac1{m(m+1)} \int_{\mathcal P_\xi} \tilde x^i d\sigma_\xi \nonumber\\
&=& \frac1{m(m+1)} \int_{\partial \mathcal P_\xi} \tilde x^i \sigma_\xi - 
\frac1{m} \int_{\mathcal P_\xi} (\tilde x^i - q^i) d\tilde x.\label{q5}
\end{eqnarray}
It follows from \eqref{q4} and \eqref{q5} that
\begin{equation}\label{q6}
\int_{\partial P_\xi}\tilde x^i \sigma_\xi - 
\frac{\int_{\partial \mathcal P_\xi}\sigma_\xi}{\int_{\mathcal P_\xi}d\tilde x} \int_{\mathcal P_\xi} \tilde x^i\, d\tilde{x}
 = -(m+1) \int_{\mathcal P_\xi} \tilde x^i d\tilde x + (m+1) q^i.
\end{equation}
Comparing \eqref{q2.1} and \eqref{q6} we see that $\Fut$ vanishes if and only if the barycenter of 
$\mathcal P_\xi$ lies at the
distinguished point $q$. By the affine map $\Phi$, the barycenter of $\mathcal P_\xi$ is mapped to the barycenter
of $\frac1{m+1}\mathcal P_{-K_S}$. Therefore, in view of \eqref{Fut4}, $q$ is mapped to the origin $0$ in $\frac1{m+1}\mathcal P_{-K_S}$.
\end{proof}

\section{Volume minimization}\label{section5}

 
 In this section we prove Theorem \ref{thm1}.
 Let $\mathcal C$ be one of $\mathcal C_\xi$ or $\mathcal C_{\xi,\alpha}$'s , and $\mathcal C^\ast$ be
 $$ \mathcal C^\ast = \{ y \in \mathfrak t\ |\ \langle y,p \rangle \ge 0\ \ \text{for all}\ p \in \mathcal C\}.$$
 For a Reeb vector field $\xi$ we put
 $$ \mathcal P_\xi = \{ p \in \mathcal C \subset \mathfrak t^\ast\ |\ \langle p, \xi \rangle = 1\}.$$
Let $q \in P_\xi$ be fixed (we have $q$ in the proof of Theorem \ref{thm3} in mind which is also equal to
the origin $\bfo$ in the statement of Theorem \ref{thm3}), and put
$$ \Xi_q = \{ \xi' \in \mathcal C^\ast \subset \mathfrak t\ |\ \langle \xi',q \rangle = 1 \}.$$
We regard $\Xi_q$ as the space of Reeb vector fields, and look for $\xi' \in \Xi_q$ which minimizes
the volume functional defined as follows. For $\xi' \in \Xi_q$ we define
\begin{equation}\label{Stokes0}
\mathcal P_{\xi'} = \{ p \in \mathcal C\ |\ \langle \xi',p \rangle = 1 \},
\end{equation}
$$ \Delta_{\xi'} = \{ p \in \mathcal C\ |\ \langle p, \xi \rangle \le 1\}$$
and the volume functional $ \Vol : \Xi_q \to \bfR$
by
\begin{eqnarray}
 \Vol(\xi') &:=& \frac1{|\xi'|}\Vol (\mathcal P_{\xi'}) \nonumber\\
 &=& (m+1) \Vol(\Delta_{\xi'}). \label{Stokes1}
 \end{eqnarray}

\begin{proposition}\label{prop1}
Let $\xi_t = \xi' + t\nu$ be a path in $\Xi_q$. Then
$$
\left.\frac{d}{dt}\Vol(\xi_t)\right|_{t=0} = - \frac{m+1}{|\xi'|} \int_{\mathcal P_{\xi'}} \langle p,\nu \rangle d\sigma
$$
where $d\sigma$ is the natural measure on $\mathcal P_{\xi'}$ induced from $\mathcal P_{\xi'} \subset
\mathfrak t^\ast \cong \bfR^{m+1}$ where the last isomorphism is induced from the property \eqref{good} of 
good convex polyhedral cone.
\end{proposition}
\begin{proof}
Let $\xi'$ be in $\Xi_q$ and 
$p$ be in $P_{\xi'}$ so that $\langle p,\xi' \rangle = 1$. For the path $\xi_t = \xi' + t\nu$ in $\Xi_q$
we have $\langle q, \nu \rangle = 0$. Write $p' \in P_{\xi_t}$ as $p' = sp$. Then
\begin{eqnarray*}
s &=& \frac 1{1 + t\langle p,\nu \rangle}\\
&=& 1 - t\langle p,\nu \rangle + t^2 (\langle p,\nu \rangle)^2 + O(t^3).
\end{eqnarray*}
Thus
$$
\left.\frac{d}{dt}\right|_{t=0} d(sp_1)\wedge \cdots \wedge d(sp_{m+1})
= -(m+2)\langle p, \nu \rangle dp_1 \wedge \cdots \wedge dp_{m+1}
$$
From this 
$$
\left.\frac{d}{dt}\Vol(\Delta_{\xi_t})\right|_{t=0} = - (m+2)\int_{\Delta_{\xi'}} \langle p,\nu \rangle dp_1 \wedge \cdots
\wedge dp_{m+1}.
$$
Using the Stokes theorem we have
\begin{eqnarray*}
(m+2)\int_{\Delta_{\xi'}} p_i dp_1 \wedge \cdots \wedge dp_{m+1} 
&=& \int_{\Delta_{\xi'}} (\sigma_j \frac{\partial}{\partial p_j} (p_j p_i)) dp_1 \wedge \cdots \wedge dp_{m+1}\\
&=& \frac1{|\xi'| } \int_{P_{\xi'}} p_i d\sigma
\end{eqnarray*}
Note that \eqref{Stokes1} can also be proved similarly. Thus we obtain
\begin{eqnarray*}
\left.\frac{d}{dt}\Vol(\xi_t)\right|_{t=0} &=& (m+1)\left.\frac{d}{dt}\Vol(\Delta_{\xi_t})\right|_{t=0}\\
&=& -(m+1)(m+2)\int_{\Delta_{\xi'}} \langle p, \nu \rangle dp_1 \wedge \cdots \wedge dp_{m+1}\\
&=& -\frac{m+1}{|\xi'|} \int_{P_{\xi'}} \langle p, \nu \rangle d\sigma.
\end{eqnarray*}
This completes the proof of Proposition \ref{prop1}.
\end{proof}
\begin{proposition}\label{prop2}
In the same situation as in Proposition \ref{prop1}
$$
\left.\frac{d^2}{dt^2}\Vol(\xi_t)\right|_{t=0} = 
\frac{(m+1)(m+2)}{|\xi'|} \int_{\mathcal P_{\xi'}} \langle p,\nu \rangle^2 dp_1 \wedge \cdots \wedge dp_{m+1},
$$
\end{proposition}
\begin{proof}
By similar computations as in the proof of the previous proposition we obtain the
desired equality from
$$
\left.\frac{d^2}{dt^2}\Vol(\xi_t)\right|_{t=0} = 
(m+1)(m+2)(m+3)\int_{\Delta_{\xi'}} \langle p,\nu \rangle^2 d\sigma
$$
and
$$
(m+3)\int_{\Delta_{\xi'}} p_i p_j dp_1 \wedge \cdots \wedge dp_{m+1} 
= \frac1{|\xi'| } \int_{P_{\xi'}} p_i p_j d\sigma.
$$
This completes the proof of Proposition \ref{prop2}.
\end{proof}

\begin{proof}[Proof of Theorem \ref{thm1}]
Let us take $\mathcal C$ in Proposition \ref{prop1} and \ref{prop2}
to be one of $\mathcal C_{\xi,1}, \cdots, \mathcal C_{\xi,k}$,
and 
consider the functional 
$W :  \Xi_q \to \bfR$ defined by
$$ W(\xi') = \sum_{\alpha=1}^k \log \frac{\Vol(\mathcal P_{\xi',\alpha})}{|\xi'|}.$$
This means that we take $\mathcal P_{\xi'}$ in \eqref{Stokes0} to be 
$$\mathcal P_{\xi',\alpha}:= \{p \in \mathcal C_{\xi,\alpha}\ |\ \langle p, \xi' \rangle = 1\}$$
and apply the subsequent computations of Proposition \ref{prop1} and \ref{prop2}.
For a path $\xi_t = \xi' + t\nu$, 
we obtain by Proposition \ref{prop1}, Proposition \ref{prop2} and Schwarz inequality,
$$ \left.\frac{d^2}{dt^2}W(\xi_t)\right|_{t=0}
\ge (m+1) \sum_{\alpha=1}^k 
\frac{\int_{\mathcal P_{\xi',\alpha}} \langle p,\nu \rangle^2 dp_1 \wedge \cdots \wedge dp_{m+1}}
{\Vol(\mathcal P_{\alpha,\xi'})}.$$
This shows that $W$ is strictly convex. This completes the proof of (1).

Suppose that we have a 
 critical point $\xi'$ at which the Minkowski sum
we have the Minkowski sum decomposition
$$ \mathcal P_{\xi'} = \mathcal P_{\xi',1} + \cdots + \mathcal P_{\xi',k}$$
holds. 
Then Proposition \ref{prop1}
shows that 
$$
\sum_{\alpha=1}^k \frac{\int_{\mathcal P_{\xi',\alpha}} \langle p-q,\nu \rangle d\sigma}{\Vol(\mathcal P_{\xi',\alpha})} = 0
$$
for any $\nu$ with $\langle q, \nu \rangle = 0$.
This implies that
the sum of the barycenters of $P_{\xi',\alpha}$'s lie at $q$, which is the origin $\bfo$. 
Since  the Minkowski sum decomposition 
$ \mathcal P_{\xi'} = \mathcal P_{\xi',1} + \cdots + \mathcal P_{\xi',k}$
corresponds to the decomposition of the positive basic first Chern class $c_1^B(S,\xi')$
$$ 2\pi c_1^B(S,\xi') = \gamma'_1 + \cdots + \gamma'_k$$
with respect to $\xi'$ and the corresponding Minkowski sum decomposition of the 
transverse moment map image 
$\sum_{\alpha=1}^k \frac1{m+1}\mathcal P'_\alpha$
as in \eqref{coupledS5},
this further implies that the sum of the barycenters of the transverse moment polytopes $P'_{\alpha}$'s 
for $\xi'$ lie at 
zero. It follows from \cite{FutakiZhang18}, Theorem 1.4, that there exist transverse coupled
K\"ahler-Einstein metrics $\omega'_\alpha$ in $\gamma'_\alpha$. This completes the proof of
(2). 

In the case of $k=1$, if we take 
$\mathcal P_{\xi,1} = \mathcal P_\xi$ then $\mathcal C_{\xi,1} = \mathcal C_\xi$
(which is indeed a convex polyhedral cone of the contact toric manifold and independent
of $\xi$). Hence
$$\mathcal P_{\xi',1} = \mathcal C_{\xi} \cap \mathcal P_{\xi'} = \mathcal P_{\xi'}.$$
As shown in (2), $W$ is a strictly convex function. 
$W$ is a proper function because as $\xi'$ tends to the boundary of $\Xi_\bfo$, $\mathcal P_{\xi'}$ tends to
plane passing through $\bfo$ parallel to a facet, and the volume tends to infinity. This completes the proof of (3).
\end{proof}

With fixed Reeb vector field $\xi$, any Reeb vector field $\xi' \in \mathfrak t^\ast$ such that $\langle \bfo, \xi' \rangle = 1$ determines a toric Sasaki structure satisfying the Calabi-Yau condition of the K\"ahler cone.
As was remarked in the introduction,
even if we assume 
$ c_1^B(S) = (\gamma_1 + \cdots + \gamma_k)/2\pi$
and 
$\mathcal P_\xi = \mathcal P_{\xi,1} + \cdots + \mathcal P_{\xi,k}$
we do not in general obtain
\begin{equation}\label{Mink1.1} 
\mathcal P_\xi' = \mathcal P_{\xi',1} + \cdots + \mathcal P_{\xi',k}
\end{equation}
for other $\xi' \in \Xi_\bfo$,
nor a decomposition of the basic first Chern class $c_1^B(S,\xi')$ with respect to $\xi'$ in the form
\begin{equation}\label{Mink2.1}
 2\pi c_1^B(S,\xi') = \gamma'_1 + \cdots + \gamma'_k.
 \end{equation}
 The failure of getting a Minkowski sum decomposition \eqref{Mink1.1} can be seen from the non-linearity of
 the $CR$ $f$-twist of Apostolov-Calderbank \cite{ApoCal21MA},
see also \cite{Viza de Souza21}. 
For this we use Lemma \ref{Sasaki2}.

If $\xi' \in \mathfrak t$ is another Reeb vector field then by Lemma \ref{Sasaki2} there is a positive
Killing potential $f$ of $\xi'$ with respect to $\xi$ satisfying \eqref{Sasaki3}.
This implies 
$$ \eta_{\xi'} = \eta_D(\xi')^{-1}\eta_D = \frac 1f \eta_\xi.$$
If $x^1, \cdots, x^n$ and $x^{\prime1}, \cdots , x^{\prime n}$ are affine coordinates in terms of a basis of $\mathfrak t$
on $\mathcal P_\xi$ and $\mathcal P_{\xi'}$ respectively such that $\bfo$ is $(0, \cdots, 0)$ in both of the coordinates
then
$$ x^{\prime i} = \frac{x^i}f.$$
Let $\widetilde{\mathcal P}_{\xi}$ and $\widetilde{\mathcal P}_{\xi,\alpha}$ be the $f$-twist of 
$\mathcal P_{\xi}$ and $\mathcal P_{\xi,\alpha}$. 
If $\bfx = \bfx_1 + \cdots + \bfx_k$ for some $\bfx_\alpha \in \mathcal P_{\xi,\alpha}$ the $f$-twist of 
$\bfx$ is 
\begin{eqnarray*}
\widetilde{\bfx} &=& \frac {\bfx} {f(\bfx)}\\
&=& \frac {\bfx_1 + \cdots + \bfx_k} {f(\bfx)}\\
&\neq& \sum_{i=1}^k  \frac {\bfx_i} {f(\bfx_i)}\\
&=& \widetilde{\bfx}_1 + \cdots + \widetilde{\bfx}_k
\end{eqnarray*}
The last inequality explains the failure of getting a 
Minkowski sum 
\begin{equation*}
\widetilde{\mathcal P}_\xi' =\widetilde{\mathcal P}_{\xi',1} + \cdots + \widetilde{\mathcal P}_{\xi',k}
\end{equation*}
by the $f$-twist.


\end{document}